\documentclass[12pt,reqno]{amsart}
\usepackage{amsfonts, url}
\usepackage{amsmath,amssymb}
\def\marg#1#2{\def\marginnotetextwidth{\the\textwidth}\marginnote{\bf #1}{\bf #2}}

\numberwithin{equation}{section}
\newtheorem{theorem}{Theorem}[section]
\newtheorem{theorem*}{Theorem}

\newtheorem{lemma}[theorem]{Lemma}

\theoremstyle{definition}

\newtheorem{example}[theorem]{Example}
\newtheorem{definition}[theorem]{Definition}
\newtheorem*{definition*}{Definition}
\newtheorem{remark}[theorem]{Remark}

\newcommand{\GLL}{{\operatorname{GL}}}

\newcommand{\ch}{{\operatorname{char}}}

\newcommand{\Hom}{{\operatorname{ Hom}}}

\newcommand{\Br}{{\operatorname{ Br}}}

\newcommand{\Brnr}{{\operatorname{ Br_{\textrm{nr}}}}}

\newcommand{\Gal}{{\operatorname{ Gal}}}

\newcommand{\Pic}{{\operatorname{ Pic}}}

\begin{document}

\title[Tori and  surfaces violating a local-to-global principle for rationality]
{Tori and surfaces violating a local-to-global principle for rationality \\
\bigskip
{\it {Tores et surfaces violant le principe local-global pour la rationalit\'e}}}

\author{Boris Kunyavski\u\i }



\address{Department of
Mathematics, Bar-Ilan University, 5290002 Ramat Gan, ISRAEL}
\email{kunyav@macs.biu.ac.il}

\begin{abstract}
We show that even within a class of varieties where the Brauer--Manin obstruction
is the only obstruction to the local-to-global principle for the existence of rational
points (Hasse principle), this obstruction, even in a stronger, base change invariant form,
may be insufficient for explaining counter-examples to the local-to-global principle for rationality.
We exhibit examples of toric varieties and rational surfaces over an arbitrary global field $k$ each of those, in
the absence of the Brauer obstruction to rationality, is rational over all completions of $k$ but is not $k$-rational.

\medskip
\noindent
{\sc {R\'esum\'e}}. Nous d\'emontrons que m\^eme dans une classe des vari\'et\'es
o\`u l'obstruction de Brauer--Manin est la seule obstruction \`a l'existence
de points rationnels (le principe de Hasse) cette obstruction, m\^eme sous
sa forme la plus forte invariante par rapport au changement de base, peut \^etre
insuffisant pour expliquer des contre-exemples au principe local-global pour la
rationalit\'e. Nous pr\'esentons des exemples de vari\'et\'es toriques et de
surfaces rationnelles sur un corps global arbitraire $k$ dont chacune
est rationnelle partout localement mais n'est pas $k$-rationnelle, en absence
d'obstruction de Brauer \`a la rationalit\'e.
\end{abstract}

\thanks{This research was supported by the ISF grant 1994/20 and accomplished
during the author's visit to the IHES (Bures-sur-Yvette).
Support of these institutions is gratefully acknowledged.}

\maketitle

\section{Introduction}

This short note is inspired by a recent result by Sarah Frei and Lena Ji \cite{FJ}
where the authors have constructed a smooth, projective, $\mathbb Q$-unirational
threefold $X$, a smooth intersection of two quadrics in $\mathbb P^5$, such that
$X_v:=X\times _\mathbb Q\mathbb Q_v$ is $\mathbb Q_v$-rational
for all places $v$ of $\mathbb Q$ but $X$ is not $\mathbb Q$-rational (the first version of \cite{FJ} was conditional
on the Birch and Swinnerton-Dyer conjecture for the Jacobian of a certain genus 2 curve, in the second version
the proof is unconditional). 
Additional properties of $X$ are the absence of the Brauer obstruction to rationality (i.e.  one has
$\Br (X\times_{\mathbb Q}K)=\Br (K)$ for all field extensions $K/\mathbb Q$),
and the existence of an integral model $\mathcal X$ of $X$ whose special fibres $\mathcal X_p$
at all odd primes $p$ are $\mathbb F_p$-rational.

Colliot-Th\'el\`ene asked the author whether one can find an example with similar properties among
algebraic tori. The first result of the present note consists in an (unconditional) affirmative answer
to this question.

\begin{theorem} \label{th:tori}
For any global field $k$ there exists a smooth, projective toric $k$-variety $X$
for which the Brauer obstruction to rationality is absent, $X_v:=X\times _kk_v$ is $k_v$-rational
for all places $v$ of $k$ but $X$ is not $k$-rational.

%If $\ch (k)=0$ and $\mathcal O$ denotes the ring of integers of $k$, $X$ has an $\mathcal O$-model
%$\mathcal X$ whose special fibres at all places $\mathfrak P$ lying over odd primes are
%$\mathcal O/\mathfrak P$-rational.
\end{theorem}

\begin{remark}
We do not impose any conditions on the reductions because of the presence of
disconnected fibres of integral models of tori at the ramified places.
The unirationality condition is satisfied automatically since any $k$-torus is
$k$-unirational \cite[8.13(2)]{Bo}. One can choose $X$ in Theorem \ref{th:tori} to be of dimension 3.
\end{remark}

The second result states that one can do better and reduce the dimension of counter-examples
to two.

\begin{theorem} \label{th:surf}
For any global field $k$ of characteristic $\ne 2$
there exists a smooth, projective, $k$-unirational, geometrically rational $k$-surface $X$
for which the Brauer obstruction to rationality is absent, $X_v:=X\times _kk_v$ is $k_v$-rational
for all places $v$ of $k$ but $X$ is not $k$-rational.
\end{theorem}

It turns out that the celebrated example of a cubic $k$-surface which is
stably $k$-rational but not $k$-rational \cite{BCTSSD} also works in our
context.

\begin{remark}
A simple-minded meaning of Theorems \ref{th:tori} and \ref{th:surf}
can be formulated as follows: even within a class of varieties where the Brauer--Manin obstruction
is the only obstruction to the local-to-global principle for the existence of rational
points (Hasse principle), this obstruction, even in a stronger, base change invariant form,
may be insufficient for explaining counter-examples to the local-to-global principle for rationality.

In more technical terms, for $X$ as in each of the above theorems we have
$$\Br(X)/\Br(k)=H^1(k,N)=0$$
where $N=\Pic (X\times_k\bar k)$ is the geometric Picard group viewed as a Galois module
(for the first equality in the above formula see, e.g., \cite[Proposition~5.4.2]{CTSk}).
Moreover, this vanishing property holds after any extension $K/k$ of the ground field.
However, in the set-up of Theorem \ref{th:tori}, there are $X$ for which there
exists a subtler obstruction:
$N$ is not a stably permutation module thus preventing $X$ from being $k$-rational
or even stably $k$-rational.

As to examples in Theorem \ref{th:surf}, $N$ is a stably permutation module but
there is another obstruction to the $k$-rationality which is of geometric nature.
Namely, $X$ is birationally $k$-equivalent to a conic bundle over the projective line
with sufficiently many degenerate fibres, and one can use deep results of
Iskovskikh \cite{Is2} which rely on the classical method of linear systems with
base points going back to Segre (see also Manin's book \cite{Ma}).

To prove the rationality over all completions, both for tori and surfaces we use
a theorem of Shafarevich from the inverse Galois theory.
\end{remark}

Explicit examples implying the statement of Theorem \ref{th:tori} are contained in
Section \ref{sec:tori}.
Section \ref{sec:surf} is devoted to the proof of Theorem \ref{th:surf}.
Section \ref{sec:prelim} contains some necessary preliminaries.

\section{Preliminaries} \label{sec:prelim}
The monographs \cite{Vo}, \cite{CTSk} and the survey \cite{MT} can serve as
general references for algebraic tori, Brauer group, and rational surfaces,
respectively. Below we collect some basic facts on tori indispensable for our
considerations.
We also recall a not very well known fact related to Shafarevich's theorem
on the realizability of solvable groups as Galois groups of extensions of global
fields.

We first fix some standing notation and recall some basic definitions.
Unless stated otherwise, throughout below $k$ is an arbitrary field,
$\bar k$ is a fixed separable closure of $k$, $\Gamma =\Gal (\bar k/k)$ is the absolute Galois group of $k$.
For a $k$-variety $X$, we shorten $X\times _k\bar k$ to $\overline X$.

\begin{remark}
To justify the above notation, note that the main two objects of our attention in this note, algebraic tori
and geometrically rational surfaces, split over a separable closure of the field of definition. For tori this
is a well-known fact for which several different proofs have been produced by Ono, Borel, Springer, Tate, Tits;
see \cite{Yu} for a nice overview. For rational surfaces this was proved by Coombes \cite{Co}.
\end{remark}

The (cohomological) Brauer group of a $k$-variety $X$ is denoted $\Br (X)$.
For a smooth non-projective variety $V$ we also consider the unramified Brauer group
$\Brnr (k(V)/k)$, which is isomorphic to $\Br (X)$ where $X$ is a smooth projective
variety containing $V$ as an open subset; see \cite{CTSk} for details.

\begin{remark} \label{Brnr}
For the classes of varieties considered in this paper, we have $\Brnr(k(V)/k)\cong
H^1(k, \Pic (\overline X))$, where $\Pic (\overline X)$ is the Picard group of $\overline X$
viewed as a Galois module, see, e.g. \cite[Propositions 5.4.2 and 6.2.7]{CTSk}.
\end{remark}

\begin{definition} \label{rat}
We say that a  $k$-variety $V$ is
\begin{enumerate}
\item[(i)] $k$-{\it rational} if $V$ is birationally $k$-equivalent to $\mathbb A^n$;
\item[(ii)] {\it stably} $k$-{\it rational} if $V\times \mathbb A^m$ is $k$-rational for some $m\ge 0$;
\item[(iii)] $\Br$-{\it trivial} if $\Brnr (K(V)/K)$ is isomorphic to $\Br (K)$ for all field extensions $K/k$.
\end{enumerate}
\end{definition}
We have irreversible implications $\mathrm{(i)\Rightarrow(ii)\Rightarrow(iii)}$.
In case (iii), we sometimes say that the Brauer obstruction is absent.

\subsection{Algebraic tori}

Given a $k$-torus $T$, the free abelian group $M=\hat T=\Hom (T,\mathbb G_{\mathrm m})$ viewed as a $\Gamma$-module
is called the character module of $T$. The category of $k$-tori is dual to the category of finitely generated
$\mathbf Z$-free $\Gamma$-modules (for the sake of brevity, we just say ``modules'' throughout below).

Every $k$-torus splits over a finite Galois field extension $L$ of $k$, i.e.
$T\times_kL$ is isomorphic to the split $L$-torus $\mathbb G_{\mathrm m,L}^d$, $d=\dim T$. The smallest such $L$
is called the splitting field of $T$, we denote $\Pi=\Gal (L/k)$ and call it the splitting group of $T$. Accordingly,
we replace $\Gamma$-modules with $\Pi$-modules.

\begin{definition} \label{mod}
We say that a $\Pi$-module $N$ is
\begin{enumerate}
\item [(i)] {\it permutation} if it has a $\mathbb Z$-base permuted by $\Pi$;
\item
[(ii)] {\it stably permutation} if $N\oplus S_1\cong S_2$ for some permutation modules $S_1$, $S_2$;
\item
[(iii)] {\it invertible} if $N$ is a direct summand of a permutation module;
\item
[(iv)] $H^1$-{\it trivial} (aka coflasque) if $H^1(\Pi',N)=0$ for all subgroups $\Pi'\le \Pi$;
 \item
[(v)] {\it flasque} if the dual module $N^{\circ}:=\Hom (N,\mathbb Z)$ is $H^1$-trivial;
\item
[(vi)] $H$-{\it trivial} if it is both flasque and coflasque.
\end{enumerate}
\end{definition}

Accordingly, we sometimes say that a torus $T$ is flasque (coflasque) if so is
its character module $\hat T$.

We have irreversible implications $\mathrm{(i)\Rightarrow(ii)\Rightarrow(iii)\Rightarrow(iv)\cap (v)=(vi)\Rightarrow(iv)({\textrm{ or}} (v))}$.

Any module $M$ can be embedded into a short exact sequence
\begin{equation} \label{flasque}
0\to M\to S\to F\to 0,
\end{equation}
where $S$ is permutation and $F$ is flasque; such a sequence is called a flasque resolution of $M$.

%By duality, any module $M$ can be embedded into a short exact sequence
%\begin{equation} \label{coflasque}
%0\to N\to S\to M\to 0,
%\end{equation}
%where $S$ is permutation and $N$ is flasque; such a sequence is called a coflasque resolution of $M$.

Note that if $M=\hat T$ is the character module of a torus $T$, we have $\Brnr (k(T)/k)\cong H^1(\Pi, F)$. 
Indeed, any torus $T$ can be embedded into a smooth projective model $X$ as an open subset 
(this is true even in positive characteristic, see \cite{CTHS}), the module $F:=\Pic (\overline X)$ is flasque 
\cite[4.6]{Vo}, \cite[Proposition~6]{CTSa}, and we can argue as in Remark \ref{Brnr}.    
More generally, the properties of $F$ encode the rationality properties of $T$
as follows. Let us add to the list given in Definition \ref{rat} one more property: we say that $T$ is

$\mathrm{(ii')}$  {\it retract rational}  if $T\times T'$ is $k$-rational for some torus $T'$.

We then have the implications $\mathrm{(i)\Rightarrow(ii)\Rightarrow(ii')\Rightarrow(iii)}$ for the
properties listed in Definition \ref{rat}, all irreversible
except possibly for the leftmost one, whose reversibility is a notoriously difficult long-standing problem.
The rightmost implication is a consequence of the following relation between the properties in Definitions \ref{rat} and \ref{mod}:
$T$ is stably rational (resp. retract rational, resp. $\Br$-trivial) if and only if the module $F$ in a flasque
resolution of $M=\hat T$ is stably permutation (resp. invertible, resp. $H$-trivial).

We shall use the notation $R_{K/k}$ for Weil's restriction of scalars from $K$ to $k$,
and in particular, the kernel of the norm map $R_{K/k}\mathbb G_{\textrm{m},K}\to \mathbb G_{\textrm{m},k}$
will be called a norm torus and denoted $R_{K/k}^1\mathbb G_{\textrm{m}}$.

\subsection{Shafarevich's theorem}
We shall systematically use the following fact contained in the
proof of the celebrated Shafarevich's theorem on the realizability
of all solvable groups as Galois groups; see \cite[Section~9.6]{NSW}
and particularly \cite{So} for details.

\begin{theorem} \label{Sh} ${\mathrm{(Shafarevich)}}$
Let $k$ be a global field, and let $G$ be a finite solvable group.
Then there exists a Galois field extension $K/k$ with group $G$
such that all decomposition groups $G_v$ are cyclic.
\end{theorem}

\section{Toric examples} \label{sec:tori}

First note that if one drops the requirement of the absence
of the Brauer obstruction, the task becomes very easy. Say,
the norm torus $R_{L/k}^1\mathbb G_{\mathrm m}$ corresponding
to a biquadratic extension $L=k(\sqrt{a},\sqrt{b})$ all of whose
decomposition groups are cyclic (which is easily achieved with
the help of the quadratic reciprocity law) fulfils the job: $T$
is not $k$-rational but all the $T_v:=T\times_kk_v$ are $k_v$-rational.
A well-known particular example is $L=\mathbf Q(\sqrt{13}, \sqrt{17})$.

In this section we exhibit three different examples of tori
with the trivial Brauer obstruction which violate the
local-global principle for rationality. In each example, we first
provide a finite solvable group $\Pi$ and a $\Pi$-module $M$
with the needed properties and then choose a Galois extension
$L/k$ of a global field $k$ with group $\Pi$ which satisfies the
conditions of Shafarevich's theorem.

\begin{example} \label{ex1}
Let $\Pi = \mathbb Z/5\mathbb Z \rtimes \mathbb Z/4\mathbb Z$,
the nontrivial semidirect product of the cyclic group of order 5 by the
cyclic group of order 4. This group is oftentimes denoted by $F_{20}$
and called the Frobenius group of degree 5. It is isomorphic to the
group of affine transformations of $\mathbb F_5$ and is a part of
the family of Frobenius groups.

Choose an element $\sigma\in\Pi$ of order 4
and define $M=\mathbb Z[\Pi/\sigma]/\mathbb Z$. If $L/k$ is a Galois
extension with group $\Pi$, then $M$ is the character
module of the norm torus $T=R_{K/k}^1\mathbb G_{\mathrm m}$ corresponding
to the extension $K/k$ where $K=L^{\sigma}$.

Let $0\to M\to S\to F\to 0$ be a flasque resolution of $M$.
It is known (see \cite[R4, (d2)]{CTSa}, \cite{EM1}, \cite{En}) that
the module $F$ is invertible but not stably permutation, so that
the torus $T$ is retract $k$-rational but not stably $k$-rational.
The absence of the Brauer obstruction follows from the retract
rationality.

Let now $L/k$ be an $F_{20}$-extension of global fields all of whose
decomposition groups $\Pi_v$ are cyclic (such an extension exists by
Shafarevich's theorem). Then every torus $T_v=T\times_kk_v$
splits over a cyclic extension of degree $2^s5^t$, and all such tori are
rational \cite{BC}.

\end{example}

\begin{remark} \label{rem:generic}
Starting from a global field $k$, it is not an obvious task to exhibit an explicit
$F_{20}$-extension $L/k$ all of whose decomposition groups are cyclic. If $\ch(k)=0$, one can
try to use the generic $F_{20}$-polynomial constructed by Lecacheux \cite{Le}.
\end{remark}

Although Example \ref{ex1} is sufficient for our goals, we present another two,
each having certain advantages.

\begin{example} \label{ex2}
The torus in Example \ref{ex1} is not so far from being rational. One can construct
an example with the same properties where $T$ is not retract rational, at the
expense of small dimension and explicit construction.

The construction is based on the work of Endo and Miyata \cite{EM2} who classified
the finite groups $\Pi$ such that there exists an $H$-trivial $\Pi$-module which is not
invertible. It turned out that this class consists of all groups
except those whose $p$-Sylow subgroups are cyclic for all odd $p$ and the $2$-Sylow subgroups
are cyclic or dihedral. Thus we can consider the simplest example $\Pi=(\mathbb Z/2\mathbb Z)^3$.
By \cite[Theorem~2.1]{EM2}, there exists an $H$-trivial $\Pi$-module $N$ which is not invertible.
Then any $k$-torus $T$ such that the module $F$ in a flasque resolution \eqref{flasque} of $M=\hat T$ is isomorphic to
$N$ is $\Br$-trivial but not retract rational.

To construct such a $T$, one can embed $F$ into a short exact sequence $$0\to M\to \mathbb Z[\Pi]^r\to F \to 0$$
with a $\mathbb Z$-free module $M$ and consider $T$ with character module $\hat T=M$.

Choosing a Galois extension $L/k$ of global fields with group $\Pi$ which
satisfies Shafarevich's theorem, we conclude that all tori $T_v$ are rational.
Indeed, any such torus is split by a quadratic extension of $k$ and is therefore
a direct product of tori of dimensions 1 and 2, hence $k$-rational \cite[Section~4.9]{Vo}.
\end{example}

\begin{remark} \label{131737}
For $k=\mathbb Q$, one can easily construct a required extension $L/\mathbb Q$ using
quadratic reciprocity, say, one can take $L=\mathbb Q(\sqrt{13},\sqrt{17}, \sqrt{89})$.
\end{remark}

\begin{remark}
The construction of the module $N$ presented above is somewhat implicit. To do this in
a more or less explicit way, one can use the construction in \cite[Section~2]{EM2}.

Namely, let $J=\mathbb Z[\Pi]/\mathbb Z$, then
one can write a flasque resolution of $J$ in the form
\begin{equation} \label{flJ}
0\to J \to \mathbb Z[\Pi]^7\to N_0 \to 0,
\end{equation}
see \cite[p.~233]{EM2}.
Let now
\begin{equation} \label{flN0}
0\to N_0\to S_0\to N_1\to 0
\end{equation}
be a flasque resolution of $N_0$.
By \cite[Lemma~2.4]{EM2}, the module $N_1$ is not invertible.

At this point, it is convenient to use the following general statement,
which might be interesting in its own right. It is implicitly contained
in a different form in \cite{EM2}, as a parenthetical note in the last
paragraph of Section 2 (without proof).

\begin{lemma} \label{lem:EM}
Every flasque $k$-torus is stably $k$-equivalent to some $\Br$-trivial
$k$-torus.
\end{lemma}

\begin{proof}
Let $T$ be a flasque torus with character module $\hat T=M$.
By \cite[Lemma~1.1(2)]{EM2}, one can embed $M$
into an exact sequence
\begin{equation} \label{cofl2}
0\to M\to M'\to S\to 0
\end{equation}
with $M'$ coflasque and $S$ permutation.
Since $M$ was supposed to be flasque,
$M'$ is also flasque, hence $H$-trivial.

It remains to notice that since $S$ is a permutation module,
sequence \eqref{cofl2} gives rise to the stable equivalence of tori $T$ and $T'$
corresponding to $M$ and $M'$, respectively. As $M'$ is $H$-trivial, $T'$ is $\Br$-trivial.
\end{proof}

We can now finish the construction.  Let us apply Lemma \ref{lem:EM}
to the flasque module $N_0$ appearing in \eqref{flJ} and \eqref{flN0}.
We obtain an $H$-trivial module $N$ such that the $k$-tori $T_0$ and $T$
corresponding to the modules $N_0$ and $N$, respectively, are stably
$k$-equivalent. Hence one can choose a flasque resolution of $N$ of the form
$$
0\to N\to S_1 \to N_1\to 0
$$
with $N_1$ the same as in \eqref{flN0}. As the module $N_1$ is not
invertible, so is $N$.  Thus $N$ is an $H$-trivial,  non-invertible $\Pi$-module.

\medskip

Note that in this example  $\dim (T)$ is huge (even at the starting point we have the
module $N_0$ of rank 49), in sharp contrast with the 4-dimensional torus
in Example \ref{ex1}.
\end{remark}

In the next example we exhibit a 3-dimensional torus with the needed properties.
One cannot do better from the point of view of dimension because all tori of
dimension 1 or 2 are rational by a theorem of Voskresenski\u\i \ \cite[Section~4.9]{Vo}.
However, the construction is a bit more complicated compared to Example \ref{ex1}.

\begin{example} \label{ex3}
The construction is based on the author's paper \cite{Ku}. As in Example
\ref{ex2}, let $\Pi=(\mathbb Z/2\mathbb Z)^3=\left<\alpha,\beta,\gamma\right>$.
Choose a subgroup of $\Pi$ of order 4, say, $\Pi_0=\left<\alpha,\beta\right>$.
Let $I:=\ker[\mathbb Z[\Pi_0]\to\mathbb Z]$ denote the augmentation
ideal of $\mathbb Z[\Pi_0]$.

Given a field $k$ and a Galois extension $L/k$ with group $\Pi$, let
$L_0=k(\sqrt{a},\sqrt{b})=L^{\gamma}$, $L_1=k(\sqrt{c})=L^{\Pi_0}$.
The extensions $L_0$ and $L_1$ are linearly disjoint and
$L=L_0L_1=k(\sqrt{a},\sqrt{b}, \sqrt{c})$ is their compositum.

Further, denote by $T_0$ the $k$-torus with character module $I$
split by $L_0$. Let $N\colon L_1\to k$ denote the norm map. We denote by the same letter
the norm map
\begin{equation} \label{norm}
N\colon R_{L_1/k}(T_0\times _kL_1)\to T_0
\end{equation}
 and define $T:=\ker N$ (which is isomorhic to the quotient $R_{L_1/k}(T_0\times _kL_1)/T_0$).

It is convenient to describe this construction translating it into the language of finite
groups of integral matrices. Indeed, given any $n$-dimensional $k$-torus $T_0$ with minimal splitting
field $L_0$, where $\Gal (L_0/k)=\Pi_0$, one can attach to the $\Pi_0$-module $\hat T$ the finite subgroup $W_0$
in $\GLL (n, \mathbb Z)$, and to isomorphic tori there correspond conjugate subgroups.
If now  $L_1$ is a quadratic extension of $k$ linearly disjoint from $L_0$, as in our example, then denoting
$L=L_0L_1$ and defining $T$ as the kernel of the map \eqref{norm}, denote by $\Pi$ the Galois group of $L/k$
and by $W\subset  \GLL(n, \mathbb Z)$ the isomorphic image of $\Pi$. Then we have $W=W_0\times \left<\pm I_n\right>$,
where $I_n$ stands for the identity $n\times n$-matrix, see \cite[Lemme~4.6]{KS}.

Applying this observation to our set-up, we conclude that the subgroup $W$ is isomorphic to $W_1$ on the list
of \cite[Theorem~1]{Ku}. In Section 4 of \cite{Ku} (see also \cite[Section~7]{HY}) it is shown that $T$ is not retract rational.
It is also noted at the end of Section 3 of \cite{Ku} (see \cite{HY} for computer verification of this fact) that
$T$ is $\Br$-trivial.

If now we choose an extension $L/k$ of global fields so that all decomposition groups are cyclic, we conclude
that all $k_v$-tori $T_v$ are rational, as in the previous example.
 \end{example}

To finish the proof of Theorem \ref{th:tori}, it remains to construct a smooth projective $k$-variety $X$
containing the torus $T$ from any of Examples \ref{ex1}, \ref{ex2}, or \ref{ex3} as a dense open subset.
Recall that such an $X$ exists for any torus defined over any field, see \cite{CTHS}. Being birationally equivalent to $T$,
the $k$-variety $X$ keeps all rationality properties of $T$. By the definition of the unramified Brauer group,
we have $\Br (X\times_kK)=\Brnr (K(T)/K)=\Br(K)$ for all extensions $K/k$. Theorem \ref{th:tori} is proven. \qed

\begin{remark}
The examples constructed above give more than stated in Theorem \ref{th:tori}, namely they provide a variety $X$
which is not only $k$-irrational but is not even stably $k$-rational.
\end{remark}

\section{Rational surfaces} \label{sec:surf}

Before proving Theorem \ref{th:surf},
recall that the story started with the affine $\mathbb Q$-variety $V\subset \mathbb A^3$
given by
\begin{equation} \label{Ts}
y^2-221z^2=(x^2-13)(x^2-17).
\end{equation}
This example first appeared in Tsfasman's PhD thesis in 1982. It was explored there
and in the subsequent papers \cite{KT}, \cite{CT} (the latter was mentioned in the introduction to
\cite{FJ}).  Among many interesting properties, a smooth projectivization $X$ of $V$
satisfies the property of being $\mathbb Q_p$-rational for all $p$, $\mathbb R$-rational
but not $\mathbb Q$-rational.  However, $\Br (X)/\Br(\mathbb Q)=\mathbb Z/2\mathbb Z$, so this counter-example
is explained by the Brauer obstruction.

The surface $X$ arising from equation \eqref{Ts} belongs to a family of Ch\^atelet surfaces. It has
a structure of conic bundle over the projective line (by projecting onto the $x$-coordinate) with
4 degenerate fibres (pairs of transversally intersecting lines)
corresponding to the zeros of the polynomial on the right-hand side.

\medskip

\noindent{\it {Proof of Theorem \ref{th:surf}}}.
The example we are going to use is somewhat similar. We start with the affine cubic surface $V\subset \mathbb A^3$
given over an arbitrary global field $k$ with $\ch (k)\ne 2$ by
\begin{equation} \label{Zar}
y^2-az^2=f(x),
\end{equation}
where $a$ is not a square in $k$ and $f\in k[x]$ is a separable irreducible polynomial of degree 3 with discriminant $a$.
This surface was explored in \cite{BCTSSD} where it was proved that it is stably $k$-rational but not $k$-rational.

The smooth projectivization $X$ of the surface $V$ given by \eqref{Zar}
is also a Ch\^atelet surface and also has 4 degenerate
fibres (at the zeros of the right-hand side and at infinity). Let $L$ denote the splitting field of $f$.
Then $\Gal (L/k)$, the Galois group of $f$,  is isomorphic to the symmetric group $S_3$, it acts on the collection
$D=\{\ell_1, \bar\ell_1,\dots ,\ell_4,\bar\ell_4\}$ of eight components of the degenerate fibres
as $G=\left<(123)c_2c_3,  (12)c_3c_4\right>$, where $(ij)$ swaps the $i^{\mathrm{th}}$ and
$j^{\mathrm{th}}$ degenerate fibres, and $c_i$ swaps the components of the $i^{\mathrm{th}}$
fibre, see \cite[Theorem~4.19]{KST}. The $\Pi$-module $\Pic (\overline X)$ is
stably permutation \cite[Th.~2]{BCTSSD}
(this is only one ingredient in the proof that these particular surfaces
$X$ are stably $k$-rational). Therefore, we have $H^1(\Pi', \Pic (\overline X))=0$
for all $\Pi'\subseteq \Pi$, so that $\Br (X\times_kK)=\Br (K)$ for all extensions $K/k$.

Suppose that all decomposition groups $G_v$ of $L/k$ are cyclic (as $S_3$ is a solvable group,
such an extension $L/k$ exists for any global field $k$). Then $G_v$ acting on $D$ is
conjugate either to a cyclic group $\left<(123)c_2c_3\right>$ of order 3, or to a cyclic group
$\left<(12)c_3c_4\right>$ of order 2. In both cases, the resulting conic bundle $k_v$-surface $X_v$
is not relatively minimal: in the first case one can blow down the $4^{\mathrm{th}}$ degenerate fibre,
and in the second case the first two ones. It remains to apply Iskovskikh's results on the structure
and birational properties of conic bundle surfaces, Namely, we conclude that

$\bullet$ each surface $X_v$ is birationally $k_v$-equivalent to a conic bundle with 2 or 3 degenerate fibres and is hence $k_v$-rational \cite[Theorem~4.1]{Is1};

$\bullet$ the surface $X$ is a relatively minimal conic bundle with 4 degenerate fibres and hence is not $k$-rational \cite[Theorem~2]{Is2}.

Theorem \ref{th:surf} is proven. \qed

\begin{remark}
The surface $X$ from Theorem \ref{th:surf} is birationally $k$-equivalent
to a Del Pezzo surface $S$ of degree $4$ with $\Pic (S)\cong \mathbb Z\oplus \mathbb Z$,
see the proof of Theorem~4.19 of \cite{KST}. In the course of this proof it is also shown that $X$ is essentially the
only example of a $\Br$-trivial $k$-irrational surface within this class
(it turned out that within this class the module $\Pic (\overline S)$ is stably permutation
if and only if it is $H$-trivial). However, in Theorem~5.20 of {\it loc. cit.}
it is proven that within the class of Del Pezzo surfaces $S$ of degree 4 with  $\Pic (S)\cong \mathbb Z$
there are three more examples of $\Br$-trivial $k$-irrational surfaces (once again,
$\Pic (\overline S)$ is stably permutation
if and only if it is $H$-trivial). All three surfaces are
$k_v$-rational if all decomposition groups of $k$ are cyclic, and
Shafarevich's theorem is also applicable in each of the three cases.
It is not known whether these surfaces are stably $k$-rational.
\end{remark}

\noindent
{\it {Acknowledgement.}} I thank Jean-Louis Colliot-Th\'el\`ene for bringing \cite{FJ}
to my attention, asking the question mentioned in the introduction, and fruitful discussions. 
I also thank the anonymous referee for useful remarks.

\enddocument